\documentclass[twoside,a4paper,reqno,11pt]{amsart} 
\usepackage{amsfonts, amsbsy, amsmath, amssymb, latexsym, mathtools}
\usepackage{mathrsfs,array}
\usepackage[top=30mm,right=30mm,bottom=30mm,left=30mm]{geometry}
\usepackage{stmaryrd}
\usepackage{bm}
\usepackage{xcolor}
\usepackage{hyperref}

\allowdisplaybreaks

\usepackage{booktabs}

\newtheorem{theorem}{Theorem}[section]
\newtheorem{conjecture}[theorem]{Conjecture}
\newtheorem{lemma}[theorem]{Lemma}
\newtheorem{proposition}[theorem]{Proposition}

{\theoremstyle{definition}
	
	\newtheorem{remark}[theorem]{Remark}

}

{\theoremstyle{definition}

}

\usepackage{mathtools}
\DeclarePairedDelimiter\ceil{\lceil}{\rceil}
\DeclarePairedDelimiter\floor{\lfloor}{\rfloor}

\newcommand{\F}{\mathbb{F}}

\begin{document}
	
\author{Melissa Lee and Tomasz Popiel}
\address{Department of Mathematics, University of Auckland, Auckland, New Zealand}
\email{melissa.lee@auckland.ac.nz, tomasz.popiel@auckland.ac.nz}

\title{Saxl graphs of primitive affine groups with sporadic point stabilisers}

\begin{abstract} 
Let $G$ be a permutation group on a set $\Omega$. 
A {\em base} for $G$ is a subset of $\Omega$ whose pointwise stabiliser is trivial, and the {\em base size} of $G$ is the minimal cardinality of a base. 
If $G$ has base size $2$, then the corresponding {\em Saxl graph} $\Sigma(G)$ has vertex set $\Omega$ and two vertices are adjacent if they form a base for $G$. 
A recent conjecture of Burness and Giudici states that if $G$ is a finite primitive permutation group with base size $2$, then $\Sigma(G)$ has the property that every two vertices have a common neighbour. 
We investigate this conjecture when $G$ is an affine group and a point stabiliser is an almost quasisimple group whose unique quasisimple subnormal subgroup is a covering group of a sporadic simple group. 
We verify the conjecture under this assumption, in all but ten cases. 
\end{abstract}

\date{\today}
\maketitle

\section{Introduction}	 \label{s:intro}

A {\em base} for a permutation group $G \leqslant \text{Sym}(\Omega)$ is a subset $B$ of $\Omega$ with the property that the pointwise stabiliser of $B$ in $G$ is trivial. 
The {\em base size} $b(G)$ of $G$ is the minimal cardinality of a base for $G$. 
Bases have been studied since the late 19th century \cite{Bochert}, with particular emphasis on {\em primitive} groups, namely transitive groups that preserve no non-trivial partition of $\Omega$, and groups with small bases. 
We refer the interested reader to the survey article \cite{BCsurvey} for details concerning the history and applications of bases. 

Note that $b(G)=1$ if and only if $G$ has a {\em regular orbit} on $\Omega$, meaning that there exists $\omega \in \Omega$ whose stabiliser $G_\omega \leqslant G$ is trivial; and $b(G)=2$ if and only if there exists $\omega \in \Omega$ such that $G_\omega$ has a regular orbit on $\Omega$. 
When $b(G)=2$, the {\em Saxl graph} of $G$ is defined to be the graph $\Sigma(G)$ with vertex set $\Omega$ and vertices $\omega_1$ and $\omega_2$ forming an edge if and only if $\{\omega_1,\omega_2\}$ is a base for $G$. 
Observe (see Lemma~\ref{lemma2.1}) that if $G$ acts transitively on $\Omega$ then it is a vertex-transitive subgroup of $\text{Aut}(\Sigma(G))$, because the image of a base under an element of $G$ is again a base, and that if $G$ acts primitively on $\Omega$ then $\Sigma(G)$ is connected, because the connected components form a system of imprimitivity.

Saxl graphs are named in honour of Jan Saxl, who initiated a project to classify the finite primitive permutation groups with base size~$2$. 
They were introduced by Burness and Giudici \cite{BG}, who made the following conjecture.

\begin{conjecture}[Burness and Giudici \cite{BG}] \label{BGconj}
If $G$ is a finite primitive permutation group with $b(G)=2$, then every two vertices of $\Sigma(G)$ have a common neighbour. 
\end{conjecture}

Burness and Giudici verified their conjecture for various classes of primitive groups, including 
\begin{itemize}
\item those of degree less than $4096$ (see their Remark~4.8), 
\item certain `large' groups of diagonal or twisted wreath type (Propositions~4.6--4.7),
\item alternating and symmetric groups with primitive point stabilisers (Theorem~5.1), 
\item several sporadic simple groups and their almost simple extensions (Theorem~6.1). 
\end{itemize}
They also demonstrated that the assumption of primitivity is essential, by constructing imprimitive groups whose  Saxl graphs are connected with arbitrarily large diameter, or disconnected with arbitrarily many connected components (Propositions~4.1 and 4.2). 

Burness and Huang~\cite{BurnessHuang} have verified Conjecture~\ref{BGconj} for almost simple groups with socle $\text{PSL}(2,q)$, and for almost simple groups with soluble point stabilisers. 
The former result built on work of Chen and Du~\cite{ChenDu}, who proved that Saxl graphs of almost simple groups with socle $\text{PSL}(2,q)$ have diameter $2$, which is a weaker condition than the conclusion of  Conjecture~\ref{BGconj}. 
In addition, Chen and Huang~\cite{ChenHuang} have calculated the valency of $\Sigma(G)$ when $G$ is an almost simple primitive group with socle an alternating group and soluble point stabiliser. 
(As noted in \cite{BG,ChenHuang}, and in our Lemma~\ref{bg_cond}, $\Sigma(G)$ has the property that every two vertices have a common neighbour if its valency is at least $|\Omega|/2$.)

Here we study Conjecture~\ref{BGconj} when $G$ is a primitive affine group with sporadic point stabiliser. 
Explicitly, we prove the following theorem.

\begin{theorem} \label{mainthm}
Let $G$ be a finite almost quasisimple group with $\textnormal{soc}(G/Z(G))$ a sporadic simple group. 
Let $V$ be a non-trivial faithful irreducible $\mathbb{F}_rG$-module, where $r$ is a prime power, and write $d = \dim_{\mathbb{F}_r}(V)$. 
Suppose further that $(G,d,r)$ does not appear in Table~$\ref{todo}$. 
Then, if the affine group $GV$ satisfies $b(GV)=2$, its Saxl graph $\Sigma(GV)$ has the property that every two vertices have a common neighbour. 
\end{theorem}

\begin{table}[!t]
\begin{tabular}{llr}
\toprule
$G$ & $d$ & $r$ \\
\midrule
$3.\text{Fi}_{22}$ & $27$ & $4$ \\		
$3.\text{Fi}_{22}$ & $54$ & $2$ \\		
$3.\text{Fi}_{22}.2$ & $54$ & $2$ \\		
$Z \times \text{Co}_1$ & $24$ & $8$ \\	
$Z \circ (2.\text{Co}_1)$ & $24$ & $9$ \\	
$Z \times \text{Co}_3$ & $23$ & $5$ \\	
$Z \times \text{Co}_3$ & $22$ & $4$ \\	
$Z \circ (6.\text{Suz})$ & $12$ & $19$ \\	
$3.\text{Suz}.2$ & $24$ & $4$ \\		
$Z \circ (3.\text{Suz})$ & $12$ & $16$ \\	
\bottomrule
\end{tabular}
\caption{The $\mathbb{F}_rG$-modules $V$ not considered in Theorem~\ref{mainthm}. 
In the first column, $Z \leqslant \mathbb{F}_r^\times$, and all possibilities for $Z$ remain open. 
}
\label{todo}
\end{table}

Here $\mathbb{F}_r$ denotes the finite field of order $r$, and $GV \leqslant \operatorname{Sym}(V)$ denotes the semidirect product of the vector space $V$ by the group $G \leqslant \text{GL}(V)$. 
Recall that $GV$ acts primitively on $V$ if and only if $G$ acts irreducibly on $V$. 
Recall also that a group $H$ is {\em almost simple} if $H_0 \leqslant H \leqslant \text{Aut}(H_0)$ for some simple group $H_0$. 
The simple group $H_0$ is the unique minimal normal subgroup of $H$, and is therefore the {\em socle} of $H$, namely the subgroup of $H$ generated by the minimal normal subgroups of $H$. 
We therefore write $H_0 = \text{soc}(H)$ in this situation. 
A group is {\em quasisimple} if it is perfect and its central quotient is simple. 
An {\em almost quasisimple} group is one that has a unique quasisimple subnormal subgroup. 

\begin{remark}
The proof of Theorem~\ref{mainthm} relies on four computational techniques outlined in Section~\ref{s:proof}, where they are labelled (T1)--(T4) for reference. 
We do not know whether the modules listed in Table~\ref{todo} are counterexamples to Conjecture~\ref{BGconj}, but we have been unable to settle the conjecture in any of these cases using any of the aforementioned techniques. 
Explicitly, we have applied the techniques labelled (T1) and (T2) in each open case, and found that they do not imply the conclusion of Theorem~\ref{mainthm}. 
The techniques labelled (T3) and (T4) seem to be computationally infeasible in all of the open cases.
\end{remark}

\begin{remark}
In some of the cases in Table~\ref{todo}, it is not even clear whether $b(GV)=2$. 
We do know, from K\"ohler and Pahlings \cite[Section~4.13]{KP}, that $b(GV)=2$ for the triple $(G,d,r) = (Z\circ (6.\text{Suz}),12,19)$. 
We also know that $b(GV)=2$ for $(G,d,r) = (3.\text{Fi}_{22},54,2)$, $(3.\text{Fi}_{22}.2,54,2)$ and $(\text{Co}_3,23,5)$. 
These cases are handled in an article of Fawcett et~al. \cite[Theorem~1.1]{FMOW}, in which $b(H)$ is determined exactly in all cases where $H$ is a covering group of an almost simple group with sporadic socle, acting on a faithful irreducible $\mathbb{F}_pH$-module with $p$ a prime dividing $|H|$. 
Although this hypothesis is more restrictive than ours, we can use \cite[Theorem~1.1]{FMOW} to check that certain groups $GV$ considered in Theorem~\ref{mainthm} do {\em not} satisfy $b(GV) = 2$. 
In particular, if $G=\langle H,Z \rangle$ for some $Z \leqslant \mathbb{F}_p^\times$ and $H$ has no regular orbit on $V$, then certainly $G$ has no regular orbit on $V$. 
More generally, the proof of Theorem~\ref{mainthm} necessarily involves determining {\em all} of the considered affine groups that do not satisfy $b(GV)=2$. 
For reference, we record this information in the following theorem. 
\end{remark}

\begin{theorem} \label{thmb=2}
Let $G$ be a finite almost quasisimple group with $\textnormal{soc}(G/Z(G))$ a sporadic simple group. 
Let $V$ be a non-trivial faithful irreducible $\mathbb{F}_rG$-module, where $r$ is a prime power, and write $d = \dim_{\mathbb{F}_r}(V)$. 
Suppose further that $(G,d,r)$ does not appear in Table~$\ref{todo}$. 
If the affine group $GV$ satisfies $b(GV)>2$, then one of the following holds.
\begin{itemize}
\item[(i)] $G = \langle H,Z \rangle$ where $H$ is a covering group of an almost simple group with sporadic socle, $r$ is a prime dividing $|H|$, $Z \leqslant \mathbb{F}_r^\times$, and $(H,d,r)$ appears in \cite[Table~1]{FMOW}. 
\item[(ii)] $(G,d,r)$ appears in Table~$\ref{b=2}$. 
\end{itemize}
Conversely, $b(GV)>2$ in each case appearing in either \textnormal{(i)} or \textnormal{(ii)}.
\end{theorem}

\begin{table}[!t]
\begin{tabular}{llr}
\toprule
$G$ & $d$ & $r$ \\
\midrule
$Z \times \text{Co}_1$ & $24$ & $4$ \\
$Z \times \text{Co}_2$ & $22$ & $4$ \\
$Z \circ (2.\text{Suz})$ & $12$ & $9$ \\
$Z \circ (3.\text{Suz})$ & $12$ & $4$ \\
$Z \circ (3.\text{J}_3)$ & $9$ & $4$ \\
$Z \circ (2.\text{J}_2)$ & $6$ & $9$ \\
$Z \circ (2.\text{J}_2)$ & $6$ & $11$ \\
$3 \times (\text{J}_2.2)$ & $12$ & $4$ \\
$Z \times \text{J}_2$ & $6$ & $4$ \\
$Z \times \text{M}_{24}$ & $11$ & $4$ \\
$7 \times \text{M}_{24}$ & $11$ & $8$ \\
\bottomrule
\end{tabular}
\quad
\begin{tabular}{llr}
\toprule
$G$ & $d$ & $r$ \\
\midrule
$Z \times \text{M}_{23}$ & $11$ & $4$ \\
$Z \circ (3.\text{M}_{22})$ & $6$ & $4$ \\
$Z \circ (3.\text{M}_{22})$ & $6$ & $16$ \\
$Z \circ (3.\text{M}_{22}.2)$ & $12$ & $4$ \\
$Z \times \text{M}_{22}$ & $10$ & $4$ \\
$Z \times (\text{M}_{22}.2)$ & $10$ & $4$ \\
$Z \circ (2.\text{M}_{12})$ & $6$ & $9$ \\
$Z \times \text{M}_{12}$ & $10$ & $4$ \\
$Z \times (\text{M}_{12}.2)$ & $10$ & $4$ \\
$Z \times \text{M}_{11}$ & $5$ & $9$ \\
$2 \times \text{M}_{11}$ & $10$ & $3$ \\
\bottomrule
\end{tabular}
\caption{The $\mathbb{F}_rG$-modules in Theorem~\ref{thmb=2}(ii). 
In the first column, $Z$ is an arbitrary subgroup of $\mathbb{F}_r^\times$, except when $G \cong Z \times \text{M}_{11}$, in which case $Z \neq 1$. 
The final entry of the table refers to the two representations of $2 \times \text{M}_{11}$ in which involutions have Brauer character value $\neq 2$.}
\label{b=2}
\end{table}

The paper is organised as follows. 
Some preliminary lemmas are collected in Section~\ref{s:prelim}, and Theorem~\ref{mainthm} is proved in Section~\ref{s:proof}. 
Theorem~\ref{thmb=2} follows from the proof of Theorem~\ref{mainthm}. 

\vspace{6pt}
\noindent {\em Acknowledgements.} 
We thank Eamonn O'Brien for several discussions, and in particular for his advice on various computations required for the proof of Theorem~\ref{mainthm}. 
We also thank Tim Burness for his helpful comments on an earlier draft of the paper.

\section{Preliminaries} \label{s:prelim}
	
Let us first fix some basic notation.  
If $G$ is a group, then $Z(G)$ denotes its centre, $o(g)$ denotes the order of $g \in G$, and $g^G$ denotes the $G$-conjugacy class of $g$. 
If $\mathbb{K}$ is a field, then $\mathbb{K}^\times$ and $\overline{\mathbb{K}}$ denote its multiplicative group and algebraic closure, respectively. 
We now collect some properties of Saxl graphs and some group-theoretic preliminaries. 

\begin{lemma}[{\cite[Lemma~2.1]{BG}}] \label{lemma2.1}
Let $G \leqslant \mathrm{Sym}(\Omega)$ be a finite transitive permutation group with point stabiliser $H$. 
Suppose that $b(G)=2$, and let $\Sigma(G)$ be the Saxl graph of $G$.
\begin{itemize}
\item[(i)] $G$ is a vertex-transitive subgroup of $\mathrm{Aut}(\Sigma(G))$.
\item[(ii)] If $G$ is primitive then $\Sigma(G)$ is connected.  
\item[(iii)] $\Sigma(G)$ has valency $c|H|$, where $c$ is the number of regular orbits of $H$ on $\Omega$.
\item[(iv)] $\Sigma(G)$ is a subgraph of $\Sigma(K)$ for every subgroup $K$ of $G$.
\end{itemize}
\end{lemma}

\begin{lemma}[{\cite[Lemma~3.6]{BG}}] \label{bg_cond}
Let $G\leqslant \mathrm{Sym}(\Omega)$ be a finite transitive permutation group with $b(G)=2$. 
If the valency of the Saxl graph $\Sigma(G)$ is at least $|\Omega|/2$, then every two vertices in $\Sigma(G)$ have a common neighbour.
\end{lemma}

Given a prime power $r$, we write $V_d(r)$ for a $d$-dimensional vector space over $\mathbb{F}_r$.

\begin{lemma} \label{sum_cond}
Let $V=V_d(r)$ for some $d \geqslant 2$ and some prime power $r$, and suppose that $G\leqslant \mathrm{GL}(V)$ has a regular orbit on $V$. 
The Saxl graph $\Sigma(GV)$ of the affine group $GV$ has the property that every two vertices have a common neighbour if and only if every vector in $V$ can be written as the sum of two vectors belonging to regular orbits of $G$ on $V$. 
\end{lemma}

\begin{proof}
First suppose that every vector in $V$ can be written as the sum of two vectors belonging to regular $G$-orbits. 
Since $GV$ acts transitively on $V$ and is, by Lemma~\ref{lemma2.1}(i), a vertex-transitive subgroup of $\mathrm{Aut}(\Sigma(GV))$, it suffices to prove that, for every non-zero $v \in V$, there exists a path of length~$2$ in $\Sigma(GV)$ from $v$ to $0\in V$. 
Write $v=v_1+v_2$, where $v_1$ and $v_2$ belong to regular $G$-orbits.
Then $\{v_1,v\}$ is a base for $GV$, because $\{0,v_2\}$ is a base and there exists $g \in GV$ mapping $\{v_1,v\}$ to $\{0, v_2\}$, namely the translation $g=-v_1$. 
Since $\{0,v_1\}$ is also a base, we have a path $0 \rightarrow v_1 \rightarrow w$. 

Now suppose that every two vertices in $\Sigma(GV)$ have a common neighbour. 
Let $v \in V$. 
Then there exists $v_1 \in V$ such that $\{0,v_1\}$ and $\{v_1,v\}$ are edges in $\Sigma(GV)$. 
Since the translation $-v_1$ is an automorphism of $\Sigma(GV)$, $\{0,v-v_1\}$ is also an edge. 
In particular, $v_1$ and $v-v_1$ belong to regular $G$-orbits, and $v$ is their sum. 
\end{proof}

\begin{lemma} \label{fieldext}
Let $V=V_d(r)$ for some $d \geqslant 2$ and some prime power $r$, and let $G\leqslant \mathrm{GL}(V)$. 
If $G$ has a base of size $b$ on $V \otimes \mathbb{F}_{r^i}$, then $G$ has a base of size at most $bi$ on $V$.
\end{lemma}

\begin{proof}
Let $B=\{v_1, \dots , v_b\}$ be a base for $G$ on $V \otimes \mathbb{F}_{r^i}$. 
Let $\{\eta_1,\dots,\eta_i\}$ be a basis for $\mathbb{F}_{r^i}$ over $\mathbb{F}_r$, and write each $v_j$ as $\sum_{k=1}^i \eta_k w_{j,k}$ for $w_{j,k} \in V$. 
Then $\{w_{j,k} : 1\leqslant j \leqslant b, 1\leqslant k\leqslant i\}$ is a base for $G$ on $V$, because every element of $G$ that stabilises this set pointwise must also stabilise $B$ pointwise.
\end{proof}

Given a finite almost simple group $H$ and $h \in H \setminus \{1\}$, we define $\alpha(h)$ to be the minimal number of $\text{soc}(H)$-conjugates of $h$ needed to generate the group $\langle \text{soc}(H), h\rangle$. 
This definition extends to an almost quasisimple group $G$ as follows: consider the almost simple group $G/Z(G)$, and define $\alpha(g) = \alpha(gZ(G))$ for $g \in G \setminus Z(G)$. 

\begin{theorem}[{\cite[Theorem 1.3]{FMOW}}] \label{alphas}
Let $H$ be an almost simple group with socle a sporadic simple group, and let $h\in H\setminus\{1\}$. Then either
\begin{itemize}
\item[(i)] $\alpha(h)=2$,
\item[(ii)] $h$ is an involution and $\alpha(h)=3$, or 
\item[(iii)] $h$ belongs to one of the $H$-conjugacy classes listed in Table~$\ref{alpha_class_tab}$, and $\alpha(h) \leqslant 6$.
\end{itemize}
\end{theorem}

\begin{table}[t!]
\begin{tabular}{@{}lllll@{}}
\toprule
$\text{soc}(H)$ & $\alpha(h)=3$ & $\alpha(h)=4$ & $\alpha(h)=5$ & $\alpha(h)=6$ \\ \midrule
$\mathrm{M}_{22}$ &  & 2B &  &  \\
$\mathrm{J}_2$ & 3A & 2A &  &  \\
$\mathrm{HS}$ & 4A & 2C &  &  \\
$\mathrm{McL}$ & 3A &  &  &  \\
$\mathrm{Suz}$ &  & 3A &  &  \\
$\mathrm{Co}_2$ &  & 2A &  &  \\
$\mathrm{Co}_1$ & 3A &  &  &  \\
$\mathrm{Fi}_{22}$ & 3A, 3B & 2D &  & 2A \\
$\mathrm{Fi}_{23}$ & 3A, 3B &  &  & 2A \\
$\mathrm{Fi}_{24}'$ & 3A, 3B &  & 2C &  \\
$\mathrm{HN}$ & 4D &  &  &  \\
$\mathrm{Ly}$ & 3A &  &  &  \\
$\mathrm{B}$ &  & 2A &  &  \\ \bottomrule
\end{tabular}
\caption{Conjugacy classes and corresponding values of $\alpha(h)$ for $h \in H$ arising in Theorem~\ref{alphas}(iii). Conjugacy class names are as in the ATLAS~\cite{ATLAS}.}
\label{alpha_class_tab}
\end{table}

If $G$ is a finite almost quasisimple group, then $g \in G$ has {\em projective prime order} if $gZ(G) \in G/Z(G)$ has prime order. 
If $V=V_d(r)$ and $G \leqslant \text{GL}(V)$, then $C_V(g)$ denotes the fixed-point space of $g$, that is, $C_V(g) = \{ v \in V : vg = v \}$. 

\begin{lemma} \label{lemma2.5}
Let $V = V_d(r)$ for some $d \geqslant 2$ and some prime power $r$, and suppose that $G \leqslant \textnormal{GL}(V)$ has at most $c$ regular orbits on $V$. 
Then
\[
|V|- c|G| \leqslant \sum_{g\in X}\frac{1}{o(g)-1}|g^G||C_V(g)|,
\]
where $X$ is a set of representatives for the conjugacy classes of projective-prime-order elements in $G$. 
Moreover, 
\begin{equation} \label{2.6secondineq}
|V|- c|G| \leqslant 
\sum_{\substack{(h,k) \in \overline{X} \times \overline{\F}_r^\times \vspace{2pt} \\ \dim C_{\overline{V}}(kh) \neq 0}} 
\frac{1}{o(h)-1} |h^H| \; r^{\dim C_{\overline{V}}(kh)},
\end{equation}
where $H = G/Z(G)$, $\overline{X}$ is a set of representatives for the conjugacy classes of prime-order elements in $H$, and $\overline{V} = V \otimes \overline{\mathbb{F}}_r$. 
\end{lemma}

\begin{proof}
Let $\mathscr{O}$ be the union of the regular orbits of $G$ on $V$. 
Every element of $V \setminus \mathscr{O}$ belongs to the fixed-point space of some non-trivial element of $G \setminus Z(G)$, so
\[
V \setminus \mathscr{O} = \bigcup_{g \in G \setminus Z(G)} C_V(g), 
\quad \text{and hence} \quad 
|V|- c|G| \leqslant \sum_{g \in G\setminus Z(G)} |C_V(g)|.
\]
The first of the asserted inequalities follows because $|C_V(g')|=|C_V(g)|$ for all $g' \in g^G$ and $C_V(g)\subseteq C_V(g^i)$ for all $i < o(g)$. 
To establish the second inequality, we follow the proof of \cite[Proposition~3.1]{crosscharpaper}. 
Let $\mathcal{H} \subset G$ be a set of coset representatives of elements of prime order in $H = G/Z(G)$. 
Then every element of $V \setminus \mathscr{O}$ must be fixed by some $g \in G$ of projective prime order. 
The fixed-point space of $g$ is an eigenspace for every element of the form $gz$, where $z \in Z(G)$, so every element of $V \setminus \mathscr{O}$ belongs to some eigenspace (not necessarily the fixed-point space) of some element of $\mathcal{H}$. 
That is,
\[
V \setminus \mathscr{O} = \bigcup_{h \in \mathcal{H}} \bigcup_{z \in Z(G)} C_V(hz),
\]
and so
\[
|V|-c|G| \leqslant 
\sum_{h \in \mathcal{H}} \sum_{z \in Z(G)} |C_V(hz)| \leqslant 
\sum_{\substack{(h,k) \in \overline{X} \times \overline{\F}_r^\times \vspace{2pt} \\ \dim C_{\overline{V}}(kh) \neq 0}} 
\frac{1}{o(h)-1} |h^H| \; r^{\dim C_{\overline{V}}(kh)},
\]
as required.
\end{proof}

Recall that a partition $\lambda = (\ell_1,\dots,\ell_s)$ of a natural number $n$ is {\em weakly decreasing} if $\ell_{i+1} \leqslant \ell_i$ for all $i<s$. 
This partition {\em dominates} another weakly decreasing partition $\mu = (m_1,\dots,m_t)$ of $n$, written $\lambda \succeq \mu$, if $\sum_{j=1}^i m_j \leqslant \sum_{j=1}^i \ell_j$ for all $i$ (where we put $m_j=0$, respectively $\ell_j=0$, if $j>s$, respectively $j>t$). 

\begin{lemma} \label{karamata}
Let $\lambda$ and $\mu$ be weakly decreasing partitions of a natural number $n$, such that $\lambda \succeq \mu$. 
Then $\sum_{i \in \lambda} r^i \geqslant \sum_{i \in \mu} r^i$ for every natural number $r\geqslant 2$. 
\end{lemma}

\begin{proof}
The function on $\mathbb{R}$ given by $x \mapsto r^x$ is convex, so the result follows from Karamata's inequality \cite[Theorem~12.4]{MR3015124}.
\end{proof}

\begin{lemma}[{\cite[p.~454]{HLS}}] \label{dimCVg}
Let $V=V_d(r)$, where $d \geqslant 2$ and $r$ is a prime power, and let $G$ be an almost quasisimple subgroup of $\mathrm{GL}(V)$. 
If $g \in G$ has projective prime order, then 
\[
\dim C_V(g) \leqslant \lfloor (1-1/\alpha(g))d \rfloor.
\]
\end{lemma}

\begin{lemma} \label{lemma2.8}
Let $V=V_d(r)$, where $d \geqslant 2$ and $r$ is a prime power, and write $\overline{V} = V \otimes \overline{\mathbb{F}}_r$. 
If $G \leqslant \mathrm{GL}(V)$ is almost quasisimple and $g \in G$ has projective prime order, then
\[
\sum_{\mathclap{\substack{k \in \overline{\F}_r^\times \vspace{2pt} \\ \dim C_{\overline{V}}(kg) \neq 0}}} 
r^{\dim C_{\overline{V}}(kg)} 
\leqslant r^{\floor{(1-1/\alpha(g))d}}+r^{\ceil{d/\alpha(g)}}.
\]
\end{lemma}

\begin{proof}
Consider the partition $\lambda = (a_1,\dots,a_s)$ of $d$ obtained by taking $a_1,\dots,a_s$ to be the dimensions of the eigenspaces of $g$ acting on $\overline{V}$, arranged in weakly decreasing order. 
Each of these eigenspaces is equal to $C_{\overline{V}}(kg)$ for some $k \in \overline{\mathbb{F}}_r^\times$, so 
\[
\sum_{\mathclap{\substack{k \in \overline{\F}_r^\times \vspace{2pt} \\ \dim C_{\overline{V}}(kg) \neq 0}}} 
r^{\dim C_{\overline{V}}(kg)} 
= \sum_{i \in \lambda} r^i.
\] 
Since $\alpha(g) = \alpha(kg)$ for all $k \in \overline{\F}_r^\times$, Lemma~\ref{dimCVg} implies that $a_1 \leqslant \lfloor (1-1/\alpha(g))d \rfloor$. 
Now consider the partition $\mu = (\lfloor (1-1/\alpha(g))d \rfloor, \lceil d/\alpha(g) \rceil)$ of $d$. 
Then $\mu \succeq \lambda$ except when $\alpha(g)=2$ and $d$ is odd, in which case $\mu \succeq \lambda$ if we redefine $\mu$ by swapping its two parts. 
In either case, Lemma~\ref{karamata} implies the asserted inequality. 
\end{proof}

Given a finite group $H$, write $i_k(H)$ for the number of elements of order $k$ in $H$, and $i_P(H)$ for the number of elements of prime order in $H$. 
Given a finite almost quasisimple group $G$, an integer $d \geqslant 2$, and a prime power $r$, define
\begin{equation} \label{def-f}
f(G,d,r)=1- \frac{\frac{1}{2}\; i_P(G/Z)\Big(r^{\lfloor(1-1/\alpha)d\rfloor}+r^{\lceil d/\alpha \rceil}\Big)+i_2(G/Z)\Big(r^{\lfloor(1-1/\alpha_2)d\rfloor}+r^{\lceil d/\alpha_2 \rceil}\Big)}{r^d},
\end{equation}
where we write $Z = Z(G)$ for brevity, and 
\begin{align*}
\alpha & := \max\{ \alpha(g) : g \in G\setminus Z \text{ and } gZ \in i_P(G/Z) \setminus i_2(G/Z) \}, \\
\alpha_2 & := \max\{ \alpha(g) : g \in G\setminus Z \text{ and } gZ \in i_2(G/Z)\}.
\end{align*}

\begin{lemma} \label{f>1/2}
Let $V=V_d(r)$, where $d \geqslant 2$ and $r$ is a prime power. 
Let $G$ be an almost quasisimple subgroup of $\textnormal{GL}(V)$, and define the function $f(G,d,r)$ as in \eqref{def-f}.
\begin{itemize}
\item[(i)] If $G$ has at most $c$ regular orbits on $V$, then $f(G,d,r) < c|G|/|V|$. 
\item[(ii)] If $f(G,d,r)\geqslant 1/2$, then every two vertices in $\Sigma(GV)$ have a common neighbour. 
\end{itemize}
\end{lemma}

\begin{proof}
Assertion~(i) follows from Lemma~\ref{lemma2.8} and the second inequality in Lemma~\ref{lemma2.5}, namely \eqref{2.6secondineq}. 
Explicitly, let $H$ denote the almost simple group $G/Z$, and let $\overline{X}_2$ be the subset of elements of order $2$ in the set $\overline{X}$ defined in Lemma~\ref{lemma2.5}. 
Then
\begin{align*}
|V| - c|G| &\leqslant 
\sum_{\substack{(h,k) \in \overline{X} \times \overline{\F}_r^\times \vspace{2pt} \\ \dim C_{\overline{V}}(kh) \neq 0}} 
\frac{1}{o(h)-1} |h^H| \; r^{\dim C_{\overline{V}}(kh)} \\
&\leqslant \sum_{h \in \overline{X}} \frac{1}{o(h)-1} |h^H| \Big(r^{\lfloor(1-1/\alpha(h))d\rfloor}+r^{\lceil d/\alpha(h) \rceil}\Big) \\
&\leqslant \sum_{h \in \overline{X}\setminus\overline{X}_2} \frac{1}{2} |h^H| \Big(r^{\lfloor(1-1/\alpha)d\rfloor}+r^{\lceil d/\alpha \rceil}\Big) + \sum_{h \in \overline{X}_2} |h^H| \Big(r^{\lfloor(1-1/\alpha_2)d\rfloor}+r^{\lceil d/\alpha_2 \rceil}\Big) \\
&= \frac{1}{2} (i_P(H) - i_2(H))\Big(r^{\lfloor(1-1/\alpha)d\rfloor}+r^{\lceil d/\alpha \rceil}\Big) + i_2(H)\Big(r^{\lfloor(1-1/\alpha_2)d\rfloor}+r^{\lceil d/\alpha_2 \rceil}\Big) \\
&< \frac{1}{2} i_P(H)\Big(r^{\lfloor(1-1/\alpha)d\rfloor}+r^{\lceil d/\alpha \rceil}\Big) + i_2(H)\Big(r^{\lfloor(1-1/\alpha_2)d\rfloor}+r^{\lceil d/\alpha_2 \rceil}\Big).
\end{align*}  
That is, $|V| - c|G| < |V| - f(G,d,r)|V|$, so $f(G,d,r) < c|G|/|V|$ as claimed. 
We prove the contrapositive of (ii). 
If $\Sigma(GV)$ does not have the property that every two vertices have a common neighbour, then Lemma~\ref{bg_cond} implies that the valency of $\Sigma(GV)$ is less than $|\Omega|/2=|V|/2$. 
Lemma~\ref{lemma2.1}(iii) says that the valency of $\Sigma(GV)$ is $m|G|$, where $m$ is the exact number of regular orbits of $G$ on $V$, so we must have $m < |V|/(2|G|)$. 
Assertion~(i) therefore implies that $f(G,d,r) < m|G|/|V| < 1/2$. 
\end{proof}

We also make the following observation.

\begin{lemma} \label{f-increasing}
Given a fixed finite almost quasisimple group $G$, the function $f(G,d,r)$ defined in \eqref{def-f} is increasing in both $d$ and $r$. 
\end{lemma}

\begin{proof}
Since $G$ is fixed, so too are $i_P(G/Z)$, $i_2(G/Z)$, $\alpha$ and $\alpha_2$. 
Since $\alpha \geqslant \alpha_2 \geqslant 2$, it suffices to check that $(r^{\lfloor(1-1/a)d\rfloor}+r^{\lceil d/a \rceil})/r^d$ is a decreasing function of $d$ and $r$ for every fixed $a \geqslant 2$. 
This function is decreasing in $d$ because $\lfloor(1-1/a)(d+1)\rfloor \leqslant \lfloor(1-1/a)d\rfloor + 1$ and $\lceil (d+1)/a \rceil \leqslant \lceil d/a \rceil + 1$, i.e.
\[
\frac{r^{\lfloor(1-1/a)(d+1)\rfloor} + r^{\lceil (d+1)/a \rceil}}{r^{d+1}} \leqslant \frac{r^{\lfloor(1-1/a)d\rfloor} + r^{\lceil d/a \rceil}}{r^d}.
\]
It is decreasing in $r$ because $d$ is strictly greater than both of $\lfloor(1-1/a)d\rfloor$ and $\lceil d/a \rceil$. 
\end{proof}

\section{Proof of Theorem~\ref{mainthm}} \label{s:proof}

We now freely use the notation of the statement of Theorem~\ref{mainthm}, and that established in Section~\ref{s:prelim}. 
Further, we write $r_0$ for the characteristic of the field $\mathbb{F}_r$ underlying $V$. 
The {\em layer} of an almost quasisimple group $G$, namely its unique quasisimple subnormal subgroup, is denoted by $E(G)$, and we write
\[
S = \text{soc}(G/Z(G)) \cong E(G)/Z(E(G)).
\] 
The names of the sporadic simple groups $S$ are as in the ATLAS~\cite{ATLAS}. 

\begin{proposition} \label{prop1}
Theorem~$\ref{mainthm}$ holds if $S \cong \mathrm{M}, \mathrm{B}, \mathrm{Th},\mathrm{Ly}, \mathrm{HN}, \mathrm{Fi}_{23}, \mathrm{Fi}_{24}' \text{ or } \mathrm{O'N}$. 
\end{proposition}

\begin{proof}
For each $G$ such that $S$ is one of the listed sporadic simple groups, the minimal degree $d_0$ of a non-trivial faithful irreducible module for $G$ in characteristic $r_0>0$ is given in \cite[Table~1]{Jansen}. 
We check that $f(G,d_0,r_0) \geqslant 1/2$ in each case, using the relevant Brauer character tables given in {\sf GAP}~\cite{GAPbc,GAP} to determine $i_P(G/Z(G))$ and $i_2(G/Z(G))$, and Theorem~\ref{alphas} to determine $\alpha(G)$ and $\alpha_2(G)$. 
The result then follows from Lemma~\ref{f>1/2}(ii). 
Note that, by Lemma~\ref{f-increasing}, it suffices to check that $f(G,d_0,r_0) \geqslant 1/2$ either for the minimal $d_0$ arising for a given $r_0$, or the minimal $r_0$ arising for a given $d_0$. 
\end{proof}

\begin{proposition} \label{prop2}
Under the hypothesis of Theorem~$\ref{mainthm}$, either $b(GV)=2$ and $\Sigma(GV)$ has the property that every two vertices have a common neighbour, or $V = V_d(r)$, where $r \leqslant R$ and the triple $(E(G),d,R)$ appears in Table~$\ref{rem1}$. 
\end{proposition}

\begin{proof}
It remains to consider the cases where $S$ is not one of the groups in Proposition~\ref{prop1}. 
We begin as for Proposition~\ref{prop1}, but now it is not always the case that $f(G,d_0,r_0) \geqslant 1/2$. 
We consult Hiss and Malle \cite{HM}, which gives the dimensions $d$ of the irreducible representations of quasisimple groups with $d \leqslant 250$, and the relevant Brauer character tables in {\sf GAP}~\cite{GAPbc,GAP}. 
For each $d$ such that there exists an irreducible $d$-dimensional representation of $G$, we calculate the largest integer $R$ such that $f(G,d,R) < 1/2$. 
Since $f$ is increasing with respect to $r$ (Lemma~\ref{f-increasing}), it remains to consider only the prime powers $r \leqslant R$. 
Similarly, once $R<2$ for a given $d$, we need not consider any~larger~$d$.  
\end{proof}

\begin{table}[t!]
\begin{tabular}{@{}ll@{}}
	\toprule
	$E(G)$ & $(d,R)$ \\
	\midrule
	$\text{M}_{11}$ & $(5,63), (9,9), (10,6), (11,5), (16,3), (24,2)$ \\							
	$\text{M}_{12}$ & $(10,9), (11,9), (15,5), (16,4), (20,3), (22,2), (29,2), (30,2), (32,2), (34,2)$ \\	
	$2.\text{M}_{12}$ & $(6,66), (10,9), (12,8), (24,2), (32,2)$ \\							
	$\text{M}_{22}$ & $(10,19), (20,5), (21,4), (34,2), (42,2), (45,2)$ \\						
	$2.\text{M}_{22}$ & $(10,19), (28,3)$ \\											
	$3.\text{M}_{22}$ & $(6,107), (12,18), (15,8), (21,4), (30,2), (42,2), (45,2)$ \\				
	$4.\text{M}_{22}$ & $(16,8), (32,2)$ \\											
	$6.\text{M}_{22}$ & $(36,2)$ \\													
	$12.\text{M}_{22}$ & $(24,4), (48,2)$ \\											
	$\text{M}_{23}$ & $(11,22), (21,4), (22,4), (44,2), (45,2)$ \\							
	$\text{M}_{24}$ & $(11,36), (22,5), (23,5), (44,2), (45,2)$ \\							
	$\text{J}_1$ & $(7,45), (14,6), (20,3), (22,3), (27,2), (31,2), (34,2)$ \\						
	$\text{J}_2$ & $(6,381), (12,24), (13,11), (14,11), (21,5), (26,4), (28,3), (36,2), (42,2)$ \\		
	$2.\text{J}_2$ & $(6,381), (12,24), (14,11), (28,3), (36,2), (41,2), (50,2)$ \\					
	$\text{J}_3$ & $(18,7), (36,2)$ \\												
	$3.\text{J}_3$ & $(9,68), (18,7), (36,2)$ \\											
	$\text{J}_4$ & $(112,2)$ \\														
	$\text{HS}$ & $(20,9), (21,6), (22,6), (49,2), (55,2), (56,2)$ \\							
	$2.\text{HS}$ & $(28,5), (56,2)$ \\												
	$\text{McL}$ & $(21,16), (22,11)$ \\												
	$3.\text{McL}$ & $(45,3)$ \\													
	$\text{He}$ & $(50,2), (51,2)$ \\													
	$\text{Ru}$ & $(28,5)$ \\														
	$2.\text{Ru}$ & $(28,5)$ \\														
	$\text{Suz}$ & $(64,4), (78,3), (110,2), (142,2)$ \\									
	$2.\text{Suz}$ & $(12,4890), (24,69)$ \\											
	$3.\text{Suz}$ & $(12,4890), (24,69), (66,4), (78,3), (132,2)$ \\							
	$6.\text{Suz}$ & $(12,4890), (24,69)$ \\											
	$\text{Co}_3$ & $(22,10), (23,10)$ \\												
	$\text{Co}_2$ & $(22,20), (23,19)$ \\												
	$\text{Co}_1$ & $(24,161)$ \\													
	$2.\text{Co}_1$ & $(24,161)$ \\													
	$\textrm{Fi}_{22}$ & $(77,4), (78,4)$ \\ 											
	$\textrm{3.Fi}_{22}$ & $(27,43)$, $(54,8)$ \\ 										
	\bottomrule
\end{tabular}
\caption{Data for Proposition~\ref{prop2}. 
}
\label{rem1}
\end{table}

Proposition~\ref{prop2} gives us a list of remaining cases to check in order to complete the proof of Theorem~\ref{mainthm}. 
That is, for each $G$ such that $E(G)$ appears in the first column of Table~\ref{rem1}, it remains to consider the faithful irreducible $\mathbb{F}_rG$-modules with $r \leqslant R$, for each pair $(d,R)$ given in the second column. 
As explained in the proof of the proposition, $R$ is simply the largest integer such that $f(G,d,R) < 1/2$, so several values of $r \leqslant R$ are readily ruled out either because (i) they are not prime powers, (ii) the $d$-dimensional representation in question is not actually realised over $\mathbb{F}_r$ (as one may verify using either the relevant character table in {\sf GAP}~\cite{GAPbc,GAP} or by referring to Hiss and Malle \cite{HM}), or (iii) the representation is realised but $b(GV) \neq 2$, i.e. $G$ has no regular orbit on $V$. 
We note these exclusions below on a case-by-case basis, rather than listing them in Table~\ref{rem1}. 

Disregarding the cases listed in Table~\ref{todo}, we then check that for each remaining triple $(E(G),d,r)$, each corresponding Saxl graph $\Sigma(GV)$ has the property that 
\begin{itemize}
\item[($*$)] every two vertices in $\Sigma(GV)$ have a common neighbour. 
\end{itemize}
(We label this property ($*$) for reference.) 
This is done using one of the four techniques (T1)--(T4) described below. 
Let $m$ denote the (exact) number of regular orbits of $G$ on $V$. 
\begin{itemize}
\item[(T1)] We construct the matrix group $G \leqslant \text{GL}(V)$ in {\sc Magma}~\cite{Magma}, usually using one of the functions \texttt{AbsolutelyIrreducibleModules} or \texttt{QuasisimpleMatrixGroup}. 
We then find a set of conjugacy class representatives of the elements of projective prime order in $G$, determine the dimensions of their eigenspaces, and calculate the sum on the right-hand side of the second inequality in Lemma~\ref{lemma2.5}, namely \eqref{2.6secondineq}. 
For the cases in which we apply this technique, this sum turns out to be less than $|V|/2$, so Lemma~\ref{lemma2.5} implies that $m \geqslant |V|/(2|G|)$. 
(If not, then putting $c = |V|/(2|G|)$ in \eqref{2.6secondineq} yields a contradiction, cf. the proof of Lemma~\ref{f>1/2}(ii).)
It then follows from Lemmas~\ref{lemma2.1}(iii) and~\ref{bg_cond} that $\Sigma(GV)$ has property~($*$). 
(Note that, for fixed $G$, $d$ and characteristic $r_0$, if this technique implies ($*$) for some $r=r_0^i$, then it also implies ($*$) for every $r=r_0^j$ with $j>i$. 
Indeed, if $\sigma(r)$ denotes the right-hand side of \eqref{2.6secondineq}, then $\sigma(r_0^j) < (r_0^j/r_0^i)^d\sigma(r_0^i)$, so if $\sigma(r_0^i) < (r_0^i)^d/2$ then $\sigma(r_0^j) < (r_0^j)^d/2$.)
\item[(T2)] This technique is a weaker variant of (T1), which we use for representations that are especially large or otherwise difficult to construct in {\sc Magma}~\cite{Magma}. 
The idea is again to check that the sum on the right-hand side of \eqref{2.6secondineq} is less than $|V|/2$, so that Lemmas~\ref{lemma2.1}(iii), \ref{bg_cond} and \ref{lemma2.5} together imply ($*$). 
However, here we calculate only an upper bound for this sum. 
More specifically, we calculate the dimensions of the fixed-point spaces of semisimple elements of projective prime order exactly, by using the Brauer character tables for $G$. 
However, for unipotent elements of projective prime order, we instead use upper bounds on the dimensions of the fixed-point spaces, obtained via Theorem~\ref{alphas} and Lemma~\ref{dimCVg}. 
\item[(T3)] We construct $G$ in {\sc Magma}~\cite{Magma} as in (T1), and verify that $m \geqslant |V|/(2|G|)$ by either (i) enumerating all of the orbits of $G$ on $V$, or (ii) finding sufficiently many regular orbits by random selection. 
Lemmas~\ref{lemma2.1}(iii) and~\ref{bg_cond} then imply ($*$). 
\item[(T4)] Here we instead deduce ($*$) from Lemma~\ref{sum_cond}. 
We construct $G$ in {\sc Magma}~\cite{Magma}, find (by random selection) a `small' set $R$ of vectors belonging to regular $G$-orbits, and check that every vector in $V$ can be expressed as a sum of two vectors in $R$. 
\end{itemize}

\begin{remark}
The {\sc Magma}~\cite{Magma} code that we used to implement techniques (T1)--(T4) is available on the first author's website~\cite{code}. 
We note that, in a few computationally intensive cases, we used {\em ad hoc} modifications of the code documented there. 
\end{remark}

We now proceed with the analysis of the cases in Table~\ref{rem1}, treating each remaining candidate for $S$ in turn (beginning at the bottom of the table). 
Except where indicated, each given argument applies to all representations associated with the triple $(E(G),d,r)$ being considered. 
(There are several instances where a given triple $(E(G),d,r)$ gives rise to two or more distinct representations, but we do not bring attention to these details unless there is a reason to do so, for example, when a certain technique (T1)--(T4) applies to one representation but not another, as, for instance, when $E(G) \cong 3.\text{J}_3$ with $d=18$.)  

Recall that $r_0$ denotes the characteristic of the field $\mathbb{F}_r$ underlying $V$.

\subsection{$S\cong\text{Fi}_{22}$}
According to Table~\ref{rem1}, it remains to consider the faithful irreducible $\mathbb{F}_rG$-modules of dimension $d$, where $G$ is a finite almost quasisimple group and either (i) $E(G)\cong\textrm{Fi}_{22}$, $d=77$ and $r \leqslant 4$, (ii) $E(G)\cong\textrm{Fi}_{22}$, $d=78$ and $r \leqslant 4$, (iii) $E(G)\cong 3.\textrm{Fi}_{22}$, $d=27$ and $r \leqslant 43$, or (iv) $E(G)\cong 3.\textrm{Fi}_{22}$, $d=54$ and $r \leqslant 8$. 
The representations in cases (i) and (ii) are realised only in characteristic $r_0=2$ or $3$, respectively. 
In each case, we use technique (T1) to show that property ($*$) holds. 
The representations in case (iii) arise only for $r$ an even power of $2$, so we need only consider $r \in \{4,16\}$.
We use (T1) to show that ($*$) holds for $r=16$. 
If $r=4$ then the corresponding Brauer character table shows that the only possibility for $G$ is $3.\text{Fi}_{22}$. 
This is the first case listed in Table~\ref{todo}; in other words, it remains open. 
Finally, the representations in case (iv) arise only for $r_0=2$, and are only irreducible if $G/Z(G) \cong \text{Fi}_{22}.2$ or $r$ is an odd power of $2$. 
We use (T1) to show that ($*$) holds for $r \in \{4,8\}$. 
The case $(E(G),d,r)=(3.\textrm{Fi}_{22},54,2)$ remains open, but here we must consider both $G \cong 3.\textrm{Fi}_{22}$ and $G \cong 3.\textrm{Fi}_{22}.2$, as indicated in Table~\ref{todo}.

\subsection{$S\cong\text{Co}_1$} 
It remains to consider the faithful irreducible $\mathbb{F}_rG$-modules of dimension $d=24$, where $G$ is a finite almost quasisimple group with $E(G) \cong \text{Co}_1$ or $2.\text{Co}_1$, and $r \leqslant 161$. 
If $E(G) \cong \text{Co}_1$ then the representations in question arise only in characteristic $r_0=2$, so it remains to consider $r \in \{2,4,\dots,128\}$. 
If $r \in \{2,4\}$ then $G$ has no regular orbit on $V$ because $|G| \geqslant |\text{Co}_1| > |V| \geqslant 4^{24}$, so $b(GV) \neq 2$. 
This information is contained in Theorem~\ref{thmb=2}. 
Explicitly, $r=2$ belongs to case (i) of the theorem, having been considered previously by Fawcett et~al.~\cite{FMOW}, while $r=4$ belongs to case (ii) and is recorded in Table~\ref{b=2}. 
For $r \in \{16,32,64,128\}$, we use technique (T2) to show that $\Sigma(GV)$ has property~($*$). 
The case $(E(G),d,r)=(\text{Co}_1,24,8)$ remains open, and is accordingly listed in Table~\ref{todo}. 

If $E(G) \cong 2.\text{Co}_1$ then the representations in question arise only for $r_0 \neq 2$. 
If $r \in \{3,5\}$ then $G$ has no regular orbit on $V$ because $|G| > |V|$, and if $r=7$ then \cite[Theorem~1.1]{FMOW} tells us that $G$ has no regular orbit on $V$. 
Again, this information is contained in Theorem~\ref{thmb=2}. 
For $r \geqslant 11$, we use (T2) to show that $\Sigma(GV)$ has property~($*$). 
The case $(E(G),d,r)=(2.\text{Co}_1,24,9)$ remains open (and appears in Table~\ref{todo}).

\subsection{$S\cong\text{Co}_2$} 
It remains to consider $E(G)\cong\text{Co}_2$ with $d \in \{22,23\}$ and $r \leqslant 19$. 
The representations with $d=22$ arise only in characteristic $r_0=2$. 
If $r \in \{2,4\}$ then $G$ has no regular orbit on $V$ because $|G|>|V|$. 
For $r \in \{8,16\}$, we use (T1) to show that $\Sigma(GV)$ has property~($*$). 
The representations with $d=23$ arise only for $r_0 \neq 2$. 
If $r = 3$ then $G$ has no regular orbit on $V$ because $|G|>|V|$. 
For $r \in \{5,7,9,11,13,17,19\}$, we use (T2) to show that $\Sigma(GV)$ has property~($*$).

\subsection{$S\cong\text{Co}_3$} 
It remains to consider $E(G)\cong \text{Co}_3$ with $r \leqslant 10$ and $d \in \{22,23\}$. 
The $22$-dimensional representations arise only for $r_0 \in \{2,3\}$. 
If $r \in \{2,3\}$ then $G$ has no regular orbit on $V$ because $|G|>|V|$. 
For $r \in \{8,9\}$, we use (T2) to show that $\Sigma(GV)$ has property~($*$). 
The case $(E(G),d,r)=(\text{Co}_3,22,4)$ remains open. 
The $23$-dimensional representations arise only for $r_0 \not \in \{2,3\}$. 
For $r = 7$, we use (T2) to establish property~($*$). 
The case $(E(G),d,r)=(\text{Co}_3,23,5)$ remains open.

\subsection{$S\cong\text{Suz}$} 

First consider $E(G) \cong \text{Suz}$. 
The $64$- and $78$-dimensional representations arise only for $r_0=3$, so it remains to consider $r=3$, in which case we use (T2) to establish ($*$), for both $d=64$ and $d=78$. 
The $110$-dimensional representations do not arise for $r=2$. 
For $(d,r)=(142,2)$, we again use (T2) to establish ($*$). 

The $24$-dimensional representations with $E(G) \cong 6.\text{Suz}$ arise only for $r_0 \geqslant 5$, and we use (T2) to establish ($*$) for each admissible $r \leqslant R = 69$. 
The $12$-dimensional representations with $E(G) \cong 6.\text{Suz}$ also arise only for $r_0 \geqslant 5$, but the Brauer character contains a third root of unity, so $r \equiv 1 \pmod 3$. 
If $r=7$ then $G$ has no regular orbit on $V$ because $|G|>|V|$, and if $r=13$ then \cite[Theorem~1.1]{FMOW} tells us that $G$ has no regular orbit on $V$. 
The case $(E(G),d,r)=(6.\text{Suz},12,19)$ remains open. 
Note that this representation does not extend to $6.\text{Suz}.2$, so there is just one entry in Table~\ref{todo} arising from this case. 
For the remaining admissible values of $r \leqslant R = 4890$, we use (T1) to establish ($*$). 

Now consider $E(G) \cong 2.\text{Suz}$. 
The $24$-dimensional representations arise only for $r_0=3$. 
If $r=3$ then $G$ has no regular orbit on $V$ because $|G|>|V|$. 
For $r \in \{9,27\}$, we use (T2) to establish ($*$). 
The $12$-dimensional representations also arise only for $r_0=3$. 
If $r \in \{3,9\}$ then $G$ has no regular orbit on $V$ because $|G|>|V|$. 
For $r \in \{27,81,243,729,2187\}$, we use (T1) to establish ($*$). 

Finally, consider $E(G) \cong 3.\text{Suz}$. 
For $(d,r)=(132,2)$, we use (T2) to establish ($*$). 
The $78$-dimensional representations arise only for $r_0 \notin \{2,3\}$. 
The $66$-dimensional representations arise only for $r_0 \neq 3$, and an irrationality in the Brauer character implies that $r \neq 2$. 
For $(d,r)=(66,4)$, we use (T2) to establish ($*$). 
The $24$-dimensional representations arise only for $r_0=2$. 
If $r=2$ then $G$ has no regular orbit on $V$ because $|G|>|V|$. 
The case $(E(G),d,r)=(3.\text{Suz},24,4)$ remains open. 
Note that the only possibility for $G$ in this case is  $3.\text{Suz}.2$, as per Table~\ref{todo}. 
For $r \in \{8,16,32,64\}$, we use (T1) to establish ($*$). 
The $12$-dimensional representations arise only for $r$ an even power of $2$. 
If $r=4$ then $G$ has no regular orbit on $V$ because $|G|>|V|$. 
For $r \in \{64,256,1024,4096\}$, we use (T1) to establish ($*$).  
The case $(E(G),d,r)=(3.\text{Suz},12,16)$ remains open. 
This representation does not extend to $3.\text{Suz}.2$, as reflected in Table~\ref{todo}.

\subsection{$S\cong\text{Ru}$} 
It remains to consider $E(G)\cong \text{Ru}$ or $2.\text{Ru}$ with $d=28$ and $r \leqslant 5$. 
For $E(G)\cong\text{Ru}$, the representations in question arise only for $r_0=2$. 
If $r=2$ then $G$ has no regular orbit on $V$ because $|G|>|V|$. 
For $r=4$, we use (T1) to establish ($*$). 
If $E(G)\cong 2.\text{Ru}$ then $r_0 \neq 2$, and the corresponding Brauer character contains a fourth root of unity. 
It therefore remains to consider $r=5$, in which case we use (T1) to establish~($*$).

\subsection{$S\cong\text{He}$}
Here we have $E(G)\cong\text{He}$ with $r=2$ and $d \in \{50,51\}$. 
The $50$-dimensional representations arise only for $r_0=7$, so it remains to consider only $(d,r)=(51,2)$, in which case we use (T1) to establish ($*$).

\subsection{$S\cong\text{McL}$}
The $45$-dimensional representations with $E(G)\cong 3.\text{McL}$ arise only for $r_0=5$, so there is nothing left to consider in this case. 
The $22$-dimensional representations with $E(G)\cong\text{McL}$ arise only for $r_0 \not \in \{3,5\}$. 
If $r=2$ then $G$ has no regular orbit on $V$ because $|G|>|V|$. 
For $r \in \{4,7,8,11\}$, we use (T1) to establish ($*$). 
The $21$-dimensional representations with $E(G)\cong\text{McL}$ arise only for $r_0 \in \{3,5\}$. 
If $r=3$ then $G$ has no regular orbit on $V$ by \cite[Theorem~1.1]{FMOW}. 
For $r \in \{5,9\}$, we use (T1) to establish ($*$).

\subsection{$S\cong\text{HS}$}
First consider the cases with $E(G)\cong 2.\text{HS}$. 
The $56$-dimensional representations do not arise for $r=2$, and the $28$-dimensional representations arise only for $r_0=5$, so for $d=56$ there is nothing left to consider, and for $d=28$ it remains to consider only $r=5$, in which case we use (T1) to establish ($*$). 
Now suppose that $E(G) \cong \text{HS}$. 
For $(d,r)=(56,2)$, we use (T1) to establish ($*$). 
The $49$- and $55$-dimensional representations do not arise for $r=2$. 
The $22$-dimensional representations arise only for $r_0 \not \in \{2,5\}$, so it remains to consider only $r=3$, in which case we establish ($*$) using (T4).
The $21$-dimensional representations arise only for $r_0=5$, and we establish ($*$) using (T1) for $r=5$. 
Finally, the $20$-dimensional representations arise only for $r_0=2$. 
If $r=2$ then $G$ has no regular orbit on $V$ because $|G|>|V|$. 
For $r \in \{4,8\}$, we use (T1) to establish~($*$).

\subsection{$S\cong \text{J}_4$} 
We use (T2) to establish ($*$) for $E(G) \cong \text{J}_4$ with $(d,r)=(112,2)$.

\subsection{$S\cong \text{J}_3$} 
Both the $18$- and the $36$-dimensional representations with $E(G)\cong\text{J}_3$ arise only for $r_0=3$, and for $d=18$ an irrationality in the Brauer character implies that $r$ must be an even power of $3$, so there is nothing left to consider in either case. 
Now suppose that $E(G)\cong3.\text{J}_3$. 
For $(d,r)=(36,2)$, we use (T2) to establish ($*$). 
When $d=18$, there are up to three distinct representations for each $r \leqslant 7$: two are conjugate under the action of an outer involution, and the third arises from fusing two $9$-dimensional representations by an outer involution. 
For the two conjugate representations, the Brauer character contains a third root of unity, and when $r_0=7$ there is a further irrationality implying that $r \neq 7$, so we are left to consider only $r=4$. 
We use (T1) to establish ($*$) in this case. 
The fused representation arises only when $r_0=2$. 
If $r=2$ then $G$ has no regular orbit on $V$ because $|G|>|V|$. 
For $r=4$, we use (T4) to establish ($*$). 
Finally, the $9$-dimensional representations arise only for $r$ an even power of $2$, so it remains to consider $r \in \{4,16,64\}$. 
If $r=4$ then $G$ has no regular orbit on $V$ because $|G|>|V|$. 
For $r=16$, we use (T4) to establish ($*$). 
For $r=64$, we use (T2).

\subsection{$S\cong\text{J}_2$} 
First consider $E(G)\cong 2.\text{J}_2$. 
The representations with $d \in \{36,41,50\}$ do not arise when $r_0=2$. 
Similarly, the $28$-dimensional representations do not arise when $r_0 \in \{2,3\}$. 
The $14$-dimensional representations arise only in odd characteristic. 
If $r=3$ then $G$ has no regular orbit on $V$ by \cite[Theorem~1.1]{FMOW}. 
For $r \in \{7,9,11\}$, we use (T1) to establish ($*$). 
For $r=5$, we use (T3). 
The $12$-dimensional representations arise only for $r_0 \not \in \{2,5\}$. 
If $r=3$ then $G$ has no regular orbit on $V$ because $|G|>|V|$. 
For $r \in \{7,9,11,13,17,19,23\}$, we use (T2) to establish ($*$). 
The $6$-dimensional representations arise only in odd characteristic. 
For each (odd) $r_0 \neq 5$, there is an irrationality in the Brauer character which precludes various values of $r$ (for example, if $r_0=3$ then $r$ must be an even power of $3$). 
For $r_0=5$, the Brauer character of $2.\text{J}_2.2$ includes an additional irrationality, so the representation of $2.\text{J}_2$ extends to include outer automorphisms only if $r$ is an even power of $5$. 
In either case, for the admissible values of $r$ with $31 \leqslant r \leqslant R = 381$, we use (T1) to establish ($*$). 
This leaves $r \in \{5,9,11,19,25,29\}$. 
If $r \in \{5,9\}$ then $G$ has no regular orbit on $V$ because $|G|>|V|$. 
If $r=11$ then $G$ has no regular orbit on $V$ by \cite[Theorem~2.2]{KP}.
Note that we record this case in part~(b) of Theorem~\ref{thmb=2}, i.e. in Table~\ref{b=2}, given that it was not considered by Fawcett et~al.~\cite{FMOW} because $11$ does not divide $|\text{J}_2|$. 
For $r \in \{19,25\}$, we use (T4) to establish ($*$). 
For $r=29$, we use (T3). 

Now consider $E(G)\cong \text{J}_2$. 
The $42$-dimensional representations do not arise for $r_0=2$. 
For $(d,r)=(36,2)$, we use (T1) to establish ($*$). 
The $28$-dimensional representations arise only for $r_0 \neq 3$, and we use (T4) to establish ($*$) when $r=2$. 
The $26$-dimensional representations arise only for $r$ an even power of $3$. 
The $21$-dimensional representations arise only in odd characteristic, but not for $r=3$. 
For $(d,r)=(21,5)$, we use (T1) to establish ($*$). 
The $14$-dimensional representations arise only for $r_0 \neq 3$, and additional irrationalities in the Brauer character leave only $r \in \{4,5,11\}$ to consider. 
For $r=4$, we use (T3) to establish ($*$). 
For $r \in \{5,11\}$, we use (T1). 
The $13$-dimensional representations arise only for $r$ an even power of $3$, so we are left to consider $r=9$, in which case we use (T1) to establish ($*$). 
The $12$-dimensional representations arise only for $r_0=2$. 
If $r=2$ then $G$ has no regular orbit on $V$ because $|G|>|V|$. 
For $r \in \{8,16\}$, we use (T1) to establish ($*$). 
If $r=4$ then $G$ has a regular orbit on $V$ only if $|Z(G)|=1$, namely $G \cong \text{J}_2.2$, and in this case we use (T4) to establish ($*$). 
Finally, the $6$-dimensional representations arise only for $r$ an even power of $2$. 
If $r=4$ then $G$ has no regular orbit on $V$ because $|G|>|V|$. 
For $r=16$, we use (T4) to establish ($*$). 
For $r \in \{64,256\}$, we use (T1).

\subsection{$S\cong\text{J}_1$} 

The representations with $d \in \{14,22,27,31,34\}$ do not arise for any of the given values of $r$. 
The $20$-dimensional representations arise only for $r_0=2$, but $G$ has no regular orbit on $V$ when $(d,r)=(20,2)$, by \cite[Theorem~1.1]{FMOW}. 
The $7$-dimensional representations arise only for $r_0=11$, and we use (T4) to establish ($*$) for $(d,r)=(7,11)$.

\subsection{$S\cong\text{M}_{24}$} 

The $45$-dimensional representations do not arise for $r_0=2$. 
For $(d,r)=(44,2)$, we use (T1) to establish ($*$). 
The $23$-dimensional representations arise only for $r_0 \not \in \{2,3\}$, and we use (T1) to establish ($*$) for $r=5$. 
The $22$-dimensional representations arise only for $r_0=3$, and we use (T4) to establish ($*$) for $r=3$.
The $11$-dimensional representations arise only for $r_0=2$. 
If $r \in \{2,4\}$ then $G$ has no regular orbit on $V$ because $|G|>|V|$. 
For $r \in \{16,32\}$, we use (T1) to establish ($*$). 
If $r=8$ then $G \cong \mathbb{F}_8^\times \times \text{M}_{24}$ or $G \cong \text{M}_{24}$. 
In the first case, we check that $G$ has no regular orbit on $V$ by enumerating the $G$-orbits in {\sc Magma}~\cite{Magma}. 
In the second case, we use (T4) to establish ($*$).

\subsection{$S\cong\text{M}_{23}$} 

The $45$-dimensional representations do not arise for $r_0=2$. 
For $(d,r)=(44,2)$, we use (T1) to establish ($*$). 
The $22$-dimensional representations arise only for $r_0 \not \in \{2,23\}$, so it remains to consider $r=3$. 
In this case, $V$ is the irreducible restriction of the $22$-dimensional module for $\text{M}_{24}$ considered above, so Lemma~\ref{lemma2.1}(iv) implies that $\Sigma(GV)$ satisfies ($*$). 
The $21$-dimensional representations arise only for $r_0=23$. 
The $11$-dimensional representations arise only for $r_0=2$. 
If $r \in \{2,4\}$ then $G$ has no regular orbit on $V$ because $|G|>|V|$. 
For $r=8$, we use (T3) to establish ($*$). 
For $r=16$, we use (T1).

\subsection{$S\cong\text{M}_{22}$}

The $24$- and $48$-dimensional representations with $E(G) \cong 12.\text{M}_{22}$ arise only if $r_0 \in \{5,11\}$, so there is nothing left to consider in these cases. 
Similarly, the $16$- and $32$-dimensional representations with $E(G) \cong 4.\text{M}_{22}$ arise only for $r_0=7$, and an irrationality in the Brauer character for $d=16$ implies that $r \neq 7$. 
The $36$-dimensional representations with $E(G) \cong 6.\text{M}_{22}$ arise only for $r_0=11$. 

The $28$-dimensional representations with $E(G) \cong 2.\text{M}_{22}$ arise only for $r_0=5$. 
The $10$-dimensional representations with $E(G) \cong 2.\text{M}_{22}$ arise only in odd characteristic, and an irrationality in the Brauer character for $r_0 \neq 7$ implies that $r \notin \{3,5,13,17,19\}$. 
This leaves $r=7$, $9$ and $11$, for which we establish ($*$) using (T4), (T3) and (T1), respectively. 

Now consider $E(G) \cong 3.\text{M}_{22}$. 
The $42$- and $45$-dimensional representations do not arise for $r=2$, and the $21$-dimensional representations do not arise for any $r \leqslant 4$. 
For $(d,r)=(30,2)$, we use (T1) to establish ($*$). 
The $15$-dimensional representations arise only for $r$ an even power of $2$, and we use (T1) to establish ($*$) for $r=4$.
The $12$-dimensional representations arise only for $r_0=2$. 
If $r \in \{2,4\}$ then $G$ has no regular orbit on $V$ by \cite[Theorem~1.1]{FMOW} and Lemma~\ref{fieldext}. 
Explicitly, \cite[Theorem~1.1]{FMOW} tells us that the base size of $3.\text{M}_{22} \leqslant G$ on $\mathbb{F}_2^{12}$ is $3$, and Lemma~\ref{fieldext} then implies that $G$ has no regular orbit on $\mathbb{F}_4^{12}$, because if it did then the base size of $3.\text{M}_{22}$ on $\mathbb{F}_2^{12}$ would be at most $2$. 
For $r \in \{8,16\}$, we use (T1) to establish ($*$). 
The $6$-dimensional representations arise only for $r$ an even power of $2$. 
If $r=4$ then $G$ has no regular orbit on $V$ because $|G|>|V|$. 
For $r=16$, we enumerate the $G$-orbits of $V$ using {\sc Magma}~\cite{Magma} and find that $G$ has no regular orbit on $V$. 
For $r=64$, we use (T1) to establish ($*$). 

Finally, consider $E(G) \cong \text{M}_{22}$. 
The $42$- and $45$-dimensional representations arise only in odd characteristic, and the $20$-dimensional representations arise only for $r_0=11$. 
For $(d,r)=(34,2)$, we use (T1) to establish ($*$). 
The $21$-dimensional representations arise only for $r_0 \not \in \{2,11\}$, so it remains to consider $r=3$, in which case we again use (T1). 
The $10$-dimensional representations arise only for $r_0=2$. 
If $r \in \{2,4\}$ then $G$ has no regular orbit on $V$ by \cite[Theorem~1.1]{FMOW} and Lemma~\ref{fieldext}. 
For $r=8$ and $16$, we establish ($*$) using (T3) and (T1), respectively.

\subsection{$S=\text{M}_{12}$} 

First consider $E(G) \cong 2.\text{M}_{12}$. 
The representations with $d \in \{24,32\}$ do not arise for $r_0=2$. 
The $12$-dimensional representations arise only for $r_0=3$, but if $r=3$ then $G$ has no regular orbit on $V$ by \cite[Theorem~1.1]{FMOW}. 
The $10$-dimensional representations arise only in odd characteristic, and an irrationality in the Brauer character implies that $r \neq \{5,7\}$. 
If $r=3$ then $G$ has no regular orbit on $V$ because $|G|>|V|$. 
For $r=9$, we use (T1) to establish ($*$). 
The $6$-dimensional representations arise only for $r_0=3$. 
If $r \in \{3,9\}$ then $G$ has no regular orbit on $V$ by \cite[Theorem~1.1]{FMOW} and Lemma~\ref{fieldext}. 
For $r=27$, we use (T4) to establish ($*$). 

Now consider $E(G) \cong \text{M}_{12}$. 
The representations with $d \in \{22,29,30,32,34\}$ do not arise for $r_0=2$. 
The $20$-dimensional representations also do not arise for $r_0=2$, and we use (T1) to establish ($*$) in the remaining case $r=3$. 
The $16$-dimensional representations arise only for $r_0 \neq 3$. 
An irrationality in the Brauer character implies that $r \neq 2$, and for $r=4$ we use (T1) to establish ($*$). 
The $15$-dimensional representations arise only for $r_0=3$, so we are left to consider $r=3$, in which case we use (T4) to establish ($*$). 
The $11$-dimensional representations arise only for $r_0 \not \in \{2,3\}$. 
We establish ($*$) for $r=5$ and $r=7$ using (T4) and (T1), respectively. 
Finally, the $10$-dimensional representations arise only for $r_0 \in \{2,3\}$. 
If $r=3$ then $G$ has no regular orbit on $V$ because $|G|>|V|$. 
If $r \in \{2,4\}$ then $G$ has no regular orbit on $V$ by \cite[Theorem~1.1]{FMOW} and Lemma~\ref{fieldext}. 
For $r \in \{8,9\}$, we use (T1) to establish ($*$).

\subsection{$S=\text{M}_{11}$} 

The $24$-dimensional representations arise only for $r_0=3$, so there is nothing left to consider in this case. 
Similarly, the $9$-dimensional representations arise only for $r_0=11$. 
The $16$-dimensional representations arise only for $r_0 \neq 3$, and an irrationality in the Brauer character for $r_0=2$ implies that $r \neq 2$. 
The $11$-dimensional representations arise only for $r_0 \not \in \{2,3\}$, and we use (T1) to establish ($*$) for $r=5$. 
The $5$-dimensional representations arise only for $r_0=3$. 
If $r=3$ then $G$ has no regular orbit on $V$ because $|G|>|V|$. 
If $r=9$ then $G$ has a regular orbit on $V$ only if $|Z(G)|=1$, namely $G \cong \text{M}_{11}$, and in this case we use (T4) to establish ($*$). 
If $r=27$, we instead use (T1). 
Finally, if $d=10$ then there are at most three distinct representations: one when $r_0=2$ or $\sqrt{-2} \not \in \mathbb{F}_r$, two when $r_0=11$, and three otherwise.
If $r=2$ then $G$ has no regular orbit on $V$ because $|G|>|V|$. 
If $r=4$, we use (T3) to establish ($*$). 
If $r=3$ then, by \cite[Theorem~1.1]{FMOW}, $G$ has no regular orbit on $V$ in the representation where involutions have Brauer character value $2$. 
In the other two representations, $G$ has a regular orbit if $G \cong \text{M}_{11}$, by \cite[Theorem~1.1]{FMOW}, but has no regular orbit if $G \cong 2 \times \text{M}_{11}$, which we check by enumerating the $G$-orbits in {\sc Magma}~\cite{Magma}. 
In the former case, we use (T3) to establish ($*$). 
For $r=5$, only one representation arises, and we again use~(T3) to establish ($*$).

\end{document}